\documentclass[a4paper,11pt,reqno,noindent]{amsart}
\usepackage[centertags]{amsmath}
\usepackage{amsfonts,amssymb,amsthm} 
\usepackage{hyperref}

 \hypersetup{
     pdfmenubar=false,        
     pdfnewwindow=true,      
     colorlinks=false,       
     linkcolor=blue,          
     citecolor=blue,        
     filecolor=magenta,      
     urlcolor=cyan           
 }
\usepackage{graphicx} 

\usepackage[english]{babel}
\usepackage{newlfont}
\usepackage{color}
\usepackage[body={15cm,21.5cm},centering]{geometry} 
\usepackage{fancyhdr}
\usepackage{esint}
\usepackage{enumerate}
\usepackage[numbers,sort]{natbib}

\usepackage[active]{srcltx}

\newtheorem{theorem}{Theorem}
\newtheorem{lemma}{Lemma}

\newtheorem{remark}{Remark}

\theoremstyle{definition}

\newcommand\ecke{\mathop{\hbox{\vrule height 7pt width .3pt depth 0pt
\vrule height .3pt width 5pt depth 0pt}}\nolimits}

\newcommand{\Em}{E_{\mathrm{m}}}
\newcommand{\e}{\varepsilon}

\newcommand{\Ihd}{I_{h,\Delta}}
\newcommand{\gd}{g_{\Delta}}
\newcommand{\ydel}{y^\Delta}

\newcommand{\bdD}{\partial_\Delta}

\newcommand{\R}{\mathbb{R}}

\newcommand{\Y}{\mathcal{Y}}
\renewcommand{\d}{\mathrm{d}}
\newcommand{\cof}{\mathrm{cof}\,}
\renewcommand{\L}{\mathcal{L}}
\newcommand{\A}{\mathcal{A}}

\newcommand{\id}{\mathrm{Id}}

\renewcommand{\H}{\mathcal{H}}

\newcommand\wto{\rightharpoonup}

\title{The shape of low energy configurations of a   thin elastic sheet with
  a single disclination}
\date{\today}
\author[H. Olbermann] {Heiner Olbermann}
\address[Heiner Olbermann]{Universit\"at Leipzig, Germany}
\email{heiner.olbermann@math.uni-leipzig.de}

\begin{document}

\begin{abstract}
We consider a geometrically fully nonlinear variational model for thin elastic sheets that contain
  a single disclination. The free elastic energy contains the thickness $h$ as a
  small parameter. We give an improvement of  a recently proved energy scaling
  law,    removing the next-to leading order terms in the lower bound. Then we prove the convergence of (almost-)minimizers of the free
  elastic energy towards the shape of a radially symmetric cone, up to Euclidean
  motions, weakly in the spaces $W^{2,2}(B_1\setminus B_\rho;\R^3)$ for every
  $0<\rho<1$, as the thickness $h$ is sent to 0.
\end{abstract}

\maketitle

\section{Introduction}
\subsection{Setup and previous work}
The present article continues a program \cite{MR3148079,HOregular,2015arXiv150907378O} to explore  thin elastic sheets with a single disclination  from the variational point of
view. The free energy that we consider consists of two parts: First, the non-convex membrane energy, that
penalizes the difference between the metric that is induced by the deformation
and the reference metric, which is the metric of the (singular) cone. Second,
the bending energy, which penalizes curvature. The bending energy contains a factor $h^2$, where the small parameter
$h$ is to be thought of as the thickness of the sheet (see equation
\eqref{eq:26} below for the definition). Choosing the 
cone as configuration, one gets infinite energy: While the membrane term
vanishes, the bending energy is infinite for this choice. Energetically, there is a competition
of the membrane and the bending terms; neither will vanish  for
configurations of low energy.

\medskip

Intuitively, it seems quite clear how  configurations of low energy should look like:   
They should be identical to the cone far away from the disclination, and near
the disclination,  there should be
some smoothing of the cone, at a scale $h$ (the only length
scale in the problem). For such configurations, one gets an energy of
$C^*h^2\log \frac1h$ plus terms of order $h^2$, where $C^*$ is an explicitly
known constant, see Lemma \ref{lem:uuperbd} below. It is natural to conjecture
that such a scaling behavior should indeed hold true for minimizers. However, a
proof of an ansatz-free lower bound with the same scaling is much more difficult
than the straightforward construction for the upper bound. 
In the literature, lower bounds for this setting have been ansatz
based \cite{lidmar2003virus,PhysRevA.38.1005,MR3179439}, or have assumed radial
symmetry \cite{MR3148079}. 

\medskip 

The idea underlying the recent proofs of ansatz-free lower bounds
\cite{HOregular,2015arXiv150907378O} is to control the  Gauss curvature (or a
linearization thereof) by interpolation between the membrane and the bending
term energy. The control over the Gauss curvature allows for a certain control
over the Gauss map (or the deformation gradient). This information  in turn yields lower bounds for the
bending energy, using an inequality of Sobolev/isoperimetric type.
For the corresponding result from \cite{2015arXiv150907378O} see equation \eqref{eq:12}
below. This lower bound  does not quite
achieve the conjectured scaling behavior, in that there exist next-to leading
order terms $O(h^2\log\log \frac1h)$ which are not present in the upper bound.

\medskip

Here, we are going to improve the results from \cite{2015arXiv150907378O} in two ways:
First, we give an improved lower bound for the elastic energy, which proves the
conjecture that the minimum of the energy is  given by $C^*h^2\log\frac1h
+O(h^2)$. The  observation that allows for this improvement is that it is
unnecessary to use interpolation to control  Gauss curvature and Gauss
map (or rather, linearized Gauss curvature and deformation gradient). It is enough to use the membrane energy alone to obtain the necessary
control, and make more efficient use of the Sobolev/isoperimetric inequality.

\medskip

Second, we use this improved lower bound to show a  statement about
the \emph{shape} of configurations that satisfy the energy bounds. We prove that
(almost-)minimizers converge to the conical deformation, up to Euclidean
motions. It is remarkable that that much information about
deformations of small energy can be obtained, considering that we are
dealing with  a highly non-convex variational problem. Hitherto, such results
had only been achieved for situations in which the energy scales like $O(h^2)$
or less \cite{MR1916989,MR2128713,hornung2011approximation}.  The results of
these papers will also play an important role in our proof.

\medskip


\subsection{Statement of results}
Let $B_1:=\{x\in\R^2:|x|<1\}$ be the sheet in the reference configuration. 
The singular cone may be described by the mapping $\ydel:B_1\to\R^3$,
\[
  \ydel(x)=\sqrt{1-\Delta^2} x+\Delta |x|e_3\,.
\]
Here, $0<\Delta<1$ is the height of the singular cone, and is determined by the deficit of the disclination at the
origin. The reference metric on $B_1$ is given by
\[
\begin{split}
  \gd(x)= & D\ydel(x)^TD\ydel(x)\\
=& (1-\Delta^2)\hat x\otimes \hat x+\Delta^2\hat x\otimes \hat x\\
=&\id_{2\times 2}-\Delta^2 \hat x^\bot\otimes \hat x^\bot\,,
\end{split}
\]
where $\hat x=x/|x|$ and $\hat x^{\bot}=(-x_2,x_1)/|x|$. The induced metric of a deformation $y\in W^{2,2}(B_1;\R^3)$
is 
\[
g_y=Dy^TDy\,.
\]
The free elastic energy   $\Ihd:W^{2,2}(B_1;\R^3)\to \R$ is defined by
\begin{equation}
\Ihd(y)= \int_{B_1}\left(|g_y-\gd|^2+h^2|D^2y|^2\right)\d \L^2\,,\label{eq:26}
\end{equation}
where $\d \L^2 $ denotes 2-dimensional Lebesgue measure.
In the paper \cite{2015arXiv150907378O}, we proved the existence of a constant
$C=C(\Delta)>0$ such that
\begin{equation}
2\pi\Delta^2h^2\left(\log\frac1h-2\log\log\frac1h-C\right)\leq \min_{y\in
  W^{2,2}(B_1;\R^3)}I_{h,\Delta}(y)\leq 2\pi\Delta^2h^2\left(\log\frac1h+C\right)\,.\label{eq:12}
\end{equation}
Our first aim in the present article is to  improve the lower bound for the free
elastic energy. The improvement consists in getting rid of the $\log\log\frac1h$
terms on the left hand side:
\begin{theorem}
\label{thm:main1}
There exist positive  constants $C_1,C_2,C_3$ that only depend on $\Delta$ with the
following property: First,
\begin{equation}
2\pi\Delta^2h^2\left(\log\frac1h-C_1\right)\leq \min_{y\in
  W^{2,2}(B_1;\R^3)}I_{h,\Delta}(y)\leq 2
\pi\Delta^2h^2\left(\log\frac1h+C_2\right)\label{eq:22}
\end{equation}
for all small enough $h>0$. Furthermore, if 
$y$  satisfies 
\begin{equation}
\Ihd(y)\leq 2\pi\Delta^2h^2(\log\frac1h+C_2)\label{eq:20}\,,
\end{equation}
then
\begin{align}
  \int_{B_1\setminus B_R}|D^2y|^2\d\L^2\leq
  &2\pi\Delta^2\log{\frac{1}{R}}+ C_3 \quad  \text{ for all }R\in(2h,1)\,,\label{eq:19}\\
  \int_{B_1}|g_y-\gd|^2\d\L^2\leq & C_3 h^2\,.\label{eq:21}
\end{align}
\end{theorem}
As a consequence of Theorem \ref{thm:main1}, we will be able to prove
convergence of (almost)-minimizers of the functional \eqref{eq:26} towards the
singular cone as $h\to 0$:
\begin{theorem}
\label{thm:coneconv}
Let $y^h\in W^{2,2}(B_1;\R^3)$ be a sequence with $\Ihd(y^h)\leq 2\pi\Delta^2h^2(\log\frac1h+C_2)$. Then up to  Euclidean motions, we have for every $0<\rho<1$,

\begin{equation}
y^h\wto \ydel \quad\text{ in }W^{2,2}(B_1\setminus B_\rho;\R^3)\,.\label{eq:35}
\end{equation}
\end{theorem}

\subsection{Scientific context}
In the proof of Theorem \ref{thm:main1} we show  a certain focusing of the elastic energy near the
disclination. Phenomena with such elastic energy focusing  are also observed in many other settings. In
particular, crumpled elastic sheets display networks of vertices and ridges. The investigation of these ``sharp''
structures in the physics community started in the mid-1990's. For a historical
account and an exhaustive list of references see
the very recommendable
overview article by Witten \cite{RevModPhys.79.643}. There has been quite some
activity in the analysis of ridge-like structures in particular, see
\cite{1996PhRvE..53.3750L,PhysRevLett.87.206105,1997PhRvE..55.1577L,PhysRevLett.78.1303,Lobkovsky01121995,MR2023444}. Energy
focusing in conical shapes
has been investigated in
\cite{MR1447150,CCMM,PhysRevLett.80.2358,Cerda08032005}.  Disclinations in thin
elastic sheets are particularly interesting as a modeling device for icosahedral
elastic structures. This is a popular model for virus capsids
\cite{PhysRevA.38.1005,lidmar2003virus} or carbon
nanocones \cite{MR2033874}, the structure one obtains when inserting a single five-valent vertex
into a graphene sheet (of otherwise six-valent vertices). The disclinations  are located at  the vertices of
the elastic icosahedra.

\medskip

In the mathematical
literature on thin elastic sheets, there have been two strands of
investigation: 
On the one hand, there are the rigorous derivations of elastic plate models from
three-dimensional finite elasticity by means of $\Gamma$-convergence (see \cite{MR1916989,MR2210909,MR2731157}). On the other hand, there has been quite some  effort to 
investigate the qualitative properties of low-energy states in the variational
formulation of elasticity, obtained through an analysis of the  scaling of the
free elastic energy with respect to the relevant parameters in the model, see
e.g.~\cite{bella2014wrinkles,MR1921161,kohn2013analysis,MR3275221}. The
present paper belongs of course to the latter group. In more detail,
rigorous scaling laws similar to the ones we prove here have been derived for a single fold
\cite{MR2358334} and for the so-called \emph{d-cone}
\cite{MR3168627,MR3102597}. The variational problems considered in these
references however are of a very special kind: The constraints on the shape of
the elastic sheet are quite restrictive, and the lower bounds use these
constraints in an essential way (see \cite{2015arXiv150907378O} for a detailed discussion). This is not the case for our setting, whence
our method of proof, which we have developed in
\cite{HOregular,2015arXiv150907378O} and which we refine here, is completely
different.

\subsection{Connection to convex integration and rigidity results}

The Nash-Kuiper Theorem \cite{MR0065993,MR0075640} states that given a two-dimensional Riemannian
manifold $(M,g)$, a \emph{short}\footnote{An immersion $y:M\to \R^3$ is short
  with respect to the metric $g$ on $M$ if for every curve $\gamma:[0,1]\to M$, the
  length of $y\circ\gamma$ is shorter (measured with the Euclidean metric on
  $\R^3$) than $\gamma$ (measured with $g$).}
immersion $y_0:M\to \R^3$, and $\e>0$, there exists an isometric immersion
$y_1\in C^1(M;\R^3)$ such that $\|y_1-y_0\|_{C^0}<\e$. This is relevant in our
context, since the leading order term in the energy \eqref{eq:26} measures the
distance of the deformation $y$ from an isometric immersion with respect to the
target metric $\gd$. By the Nash-Kuiper Theorem, there exists a vast amount of deformations $y$ that have
arbitrarily small membrane energy. A priori, these are all good candidates for
energy minimization. One needs a principle that shows that
all of these deformations are associated with large bending energy. The energy
scaling law from Theorem \ref{thm:main1}  shows that none of these maps can beat
the upper bound construction energetically. Theorem \ref{thm:coneconv} shows the
``stronger'' statement
that maps with low energy cannot look anything like the approximations of
$C^1$ isometric immersions that appear in the proof of the Nash-Kuiper
Theorem.  

The Nash-Kuiper result is an instance of \emph{convex integration}, a concept
that has been  developed systematically by Gromov \cite{MR864505}. In
particular, the theorem states that solutions to isometric immersion problems
are highly non-unique if one requires only $C^1$-regularity. In stark contrast, there is the uniqueness in the Weyl Problem: Given a
sufficiently smooth metric $g$ on $S^2$ with positive Gauss curvature, there
exists a \emph{unique} isometric immersion $y:S^2\to\R^3$ of
$C^2$-regularity. Such uniqueness is often called \emph{rigidity}. The dichotomy
of convex integration versus rigidity also appears in other contexts, such as
the Monge-Amp\`ere equation \cite{lewicka2015convex} and the incompressible Euler equation \cite{constantin1994onsager,isett2016proof}.

Concerning the uniqueness of solutions in the Weyl problem, the proof is due to Pogorelov
\cite{MR0346714}. In fact, he proved that  solutions are unique up
to Euclidean motions in the class of immersions of \emph{bounded extrinsic
  curvature}. The latter is the class of immersions for which the pull-back of
the volume form on $S^2$ under the Gauss map is a well defined signed Radon
measure. For smooth maps, this is just the measure $K\d A$, where $K$ is the
Gauss curvature and $\d A$ the volume element. We see that control over the
Gauss curvature excludes constructions in the style of Nash-Kuiper. This is also
the basic concept underlying our proof (with the  modification that we consider
a linearized version of Gauss curvature). We believe that this hints at a link
between 
questions about rigidity of surfaces and variational problems in the theory of thin elastic
sheets.

\subsection*{Notation}
For a closed line segment $\{a+t(b-a):t\in [0,1]\}\subset \R^2$, we write
$[a,b]$. For a semi-closed line segment $\{a+t(b-a):t\in (0,1]\}\subset \R^2$,
we write $(a,b]$. 
Throughout the text, we will assume the deficit of the disclination $0<\Delta<1$
to be fixed.
A statement such as ``$f\leq Cg$'' is shorthand for
``there exists a constant $C>0$ that only depends on $\Delta$ such that $f\leq
Cg$''. The value of $C$ may change within the same line. 

For $r>0$, we let $B_r=\{x\in\R^2:|x|<r\}$. The two-sphere $\{x\in\R^3:|x|=1\}$
is denoted by $S^2$.

The 
one-dimensional Hausdorff measure is denoted by $\H^1$.
\subsection*{Acknowledgments}
The author would like to thank Stefan M\"uller for very helpful discussions.

\section{Proof of Theorem \ref{thm:main1}}
As in \cite{2015arXiv150907378O}, the proof of the energy scaling law rests on
two observations. First,  by the weak formulation of the Hessian determinant, 
\begin{equation}
\sum_{i=1}^3\det D^2 y_i=(y_{,1}\cdot y_{,2})_{,12}-\frac12
  (|y_{,1}|^2)_{,22}-\frac12(|y_{,2}|^2)_{,11} \quad \text{ for } y\in C^2(B_1;\R^3)\,,
  \end{equation}
we get that the quantity $\sum_{i=1}^3\det D^2 y_i$ is close to
$\sum_{i=1}^3\det D^2 \ydel_i=\pi\Delta^2 \delta_0$ in $W^{-2,2}$, where
$\delta_0$ denotes the distribution $f\mapsto f(0)$. The expression
$\sum_{i=1}^3\det D^2 y_i$ is best thought of as the ``linearized Gauss
curvature'': For a metric of the form $g_y=\id_{2\times 2}+\e G$, the Gauss
curvature is $K=\e \sum_{i=1}^3\det D^2 y_i+O(\e^2)$.
Second, the following Sobolev/isoperimetric inequality translates estimates for
integrals of the Hessian determinant into lower bounds for boundary integrals of
the tangential part of the second derivative: 
\begin{lemma}
\label{lem:isoper}
For $v\in C^2(\overline{B_1})$ and $0\leq r\leq 1$,
\begin{equation}
\int_{\partial B_r}|D^2 v|\d\H^1\geq \left(4\pi \left|\int_{B_r}\det
    D^2v\d x\right|\right)^{1/2}\,.\label{eq:6}
\end{equation}
\end{lemma}
This inequality   has been used in the
literature in a number of places, see e.g.~\cite{MR1078998}. The proof of the statement above (including the sharp constant)  can be
found in \cite{2015arXiv150907378O}.

The main observation that allows for an improvement of the lower bound from
\cite{2015arXiv150907378O} is that we may get a lower bound for the quantity on the left hand
side in \eqref{eq:6} from the smallness of the membrane energy directly by
integrating a suitable test function against the membrane term $g_y-\gd$. In
\cite{2015arXiv150907378O}, we obtained such an estimate by interpolation instead, which
also uses the control over the bending energy. This is unnecessary, and gives
slightly worse estimates.

The following calculation indicates how to use the smallness of the membrane
term to obtain estimates on integrals of the linearized curvature. Let $\Phi\in
L^1(B_1)$ be such that $D^2\Phi$ is a vector-valued Radon measure with support
in $B_1$. Then we have for all $y\in C^2(B_1;\R^3)$:
\begin{equation}
\begin{split}
  \int_{B_1}\Big(\sum_{i=1}^3&\det D^2
  y_i(x)-\pi\Delta^2\delta_0\Big)\Phi(x)\d\L^2\\
&=\int_{B_1}\Bigg( \left(y_{,1}\cdot
  y_{,2}-\ydel_{,1}\cdot \ydel_{,2}\right)\Phi_{,12}\\
&\qquad-\frac12
  \left(|y_{,1}|^2-|\ydel_{,1}|^2\right)\Phi_{,22}-\frac12\left(|y_{,2}|^2-|\ydel_{,2}|^2\right)\Phi_{,11}\Bigg)\d\L^2\\
  &=-\frac12\int_{B_1}(g_y-\gd):\cof D^2\Phi\,\d\L^2\,.
\end{split}\label{eq:1}
\end{equation}
Here, 
\[
\cof
D^2\Phi=\left(\begin{array}{cc}\Phi_{,22}&-\Phi_{,12}\\-\Phi_{,21}&\Phi_{,11}\end{array}\right)
\]
denotes the cofactor matrix of $D^2\Phi$. Note that $\cof$ is linear on two by
two matrices, and hence $\cof D^2\Phi$ is a well defined Radon measure under our assumptions.
After these preliminary remarks, we construct the upper bound in the statement
of Theorem \ref{thm:main1}.  
It is obtained by a simple mollification of $\ydel$ on a ball of size $h$ centered
at the origin. 
\begin{lemma}
\label{lem:uuperbd}
  We have
\[
\inf_{y\in W^{2,2}(B_1;\R^3)}\Ihd(y)\leq 2\pi\Delta^2h^2\left(\log \frac1h
  +C\right)\,,
\]
where $C=C(\Delta)$ does not depend on $h$.
\end{lemma}
\begin{proof}
This is the same upper bound construction as in 
\cite{2015arXiv150907378O} (see Lemma 2 in that reference), and we will be brief. We choose $\eta\in C^\infty([0,
\infty))$ with $\eta=0$ on $[0,1/2]$, $\eta=1$ on $[1,\infty)$, 
and $|\eta'|\leq C$, $|\eta''|\leq C$. We set
\[
y_h(x)=\eta(|x|/h)\ydel(x)\,.
\]
One easily shows 
\[
\begin{array}{rlcrll}
  |g_{y_h}-\gd|&\leq C \quad &\text{ and }&\quad |D^2y_h|&\leq Ch^{-1} \quad&\text{ on }B_{h}\,,\\
  g_{y_h}-\gd&=0 \quad &\text{ and }&\quad |D^2y_h(x)|&=\Delta/|x| \quad&\text{ on }B_1\setminus
  B_{h}\,.
\end{array}
\] 
This implies
\[
\begin{split}
  \int_{B_1} |g_{y_h}-\gd|^2\d \L^2&\leq
  \int_{B_h} C\d\L^2\\
  &\leq C h^2\,,\\
 \int_{B_1} |D^2 y_h|^2\d\L^2&\leq \int_{B_1\setminus B_h}
  \frac{\Delta^2}{|x|^2}\d\L^2+\int_{B_h}\frac{C}{h^2}\d\L^2\\
  &= 2\pi\Delta^2\int_h^1 \frac{\d r}{r}+ C\\
  &= 2\pi\Delta^2\log\frac1h+ C\,.
\end{split}
\]
This implies the claim of the lemma.
\end{proof}
\begin{proof}[Proof of Theorem \ref{thm:main1}] The upper bound is proved by
  Lemma \ref{lem:uuperbd}; hence we may choose $C_2$ to be the constant from
  that lemma. 
Now it suffices to show the following: There exist $C_1,C_3$ such that if  $y\in W^{2,2}(B_1;\R^3)$ satisfies
\eqref{eq:20}, then also the lower bound in \eqref{eq:22} and \eqref{eq:19},
\eqref{eq:21} hold true. 

Let  $y\in W^{2,2}(B_1;\R^3)$ satisfy \eqref{eq:20}. By density of $C^2$ in
  $W^{2,2}$, we may assume $y\in C^2(B_1;\R^3)$ for a proof of the remaining statements.
Let $0<r<1$. Using Lemma \ref{lem:isoper}, we have for $i=1,2,3$:
\[
\frac{1}{2\pi}\int_{\partial B_r}|D^2y_i|\d\H^1\geq
\left(\frac1\pi\left|\int_{B_r}\det D^2y_i\d\L^2\right|\right)^{1/2}\,.
\]
Applying Jensen's inequality, we get
\[
\frac{1}{2\pi r}\int_{\partial B_r}|D^2y_i|^2\d\H^1
\geq\left(\frac{1}{2\pi r}\int_{\partial B_r}|D^2y_i|\d\H^1\right)^2\,.
\]
Combining these two estimates, we obtain
\[
\int_{\partial B_r}|D^2y_i|^2\d\H^1\geq\frac2r
\left|\int_{B_r}\det D^2 y_i\d\L^2\right|\,.
\]


By the triangle inequality, 

\begin{equation}
\int_{\partial B_r}|D^2y|^2\d\H^1\geq \frac2r\left|\int_{B_r}\sum_i\det D^2 y_i\d\L^2\right| \,.\label{eq:4}
\end{equation}
Now  choose $h_0=h_0(y)\in[h,2h]$ such that
\begin{equation}
\int_{\partial B_{h_0}}|g_y-\gd|^2\d\H^1\leq Ch^{-1}\int_{B_1}|g_y-\gd|^2\d\L^2\,.\label{eq:5}
\end{equation}
Choosing $h_0<R<1$ and integrating \eqref{eq:4} over the range $r\in[h_0,R]$, we get

 \begin{equation}
\label{eq:18}
\begin{split}
  \int_{B_R\setminus B_{h_0 }}&|D^2y|^2\d\L^2\\
\geq&
  2\left|\int_{h_0}^R\frac{1}{r}\left(\int_{B_r}\det D^2 y_i\d\L^2\right)\d r\right|\\
\geq &2\int_{h_0}^R\frac{\pi\Delta^2}{r}\d
r\\
&-2\left|\int_{h_0}^R\d r\int_{B_1}\d x \frac1r
  \chi_{B_r}(x)\left(\pi\Delta^2\delta_0(x)-\sum_i\det D^2y_i(x)\right)\right|\\
  = &
  2\pi\Delta^2\log\frac{R}{h_0}-
2\left|\int_{B_1}\d x \Phi(x)\left(\pi\Delta^2\delta_0(x)-\sum_i\det D^2y_i(x)\right)\right|\,,
  \end{split}
  \end{equation}
where we have used Fubini's Theorem to change the order of integration, and have
defined the test function
\[
\Phi(x):=\int_{h_0}^R\frac{1}{r}\chi_{B_r}(x)\d r=\begin{cases}\log\frac{R}{h_0}
  &\text{ if } |x|\leq h_0\\\log \frac{R}{|x|}&\text{ if }h_0<|x|\leq R\\0&\text{
    else. }\end{cases}
\] 
Now we set
\begin{equation}
\begin{split}
  A(R):=&\int_{B_1}\left(\sum_i\det D^2y_i-\pi\Delta^2\delta_0\right)\Phi(x)\d\L^2(x)\\
  =&-\frac12\int_{B_1}(g_y-\gd):\cof D^2\Phi\d\L^2(x)\,,
\end{split}\label{eq:2}
\end{equation}
where we have used \eqref{eq:1} in the second line.
An explicit computation yields
\[
\begin{split}
  D\Phi(x)=&-\frac{x}{|x|^2}\chi_{B_R\setminus B_{h_0}}(x)\\
  D^2\Phi(x)=&\left(-\id_{2\times 2}+2\hat x\otimes \hat
    x\right)|x|^{-2}\chi_{B_R\setminus B_{h_0}}(x)\\
&+|x|^{-1}\hat x\otimes \hat
  x\left(\H^1\ecke\partial B_R-\H^1\ecke\partial B_{h_0}\right)\,.
\end{split}
\]
Inserting these computations in  \eqref{eq:2}, we have

\begin{equation}
  \begin{split}
    |A(R)|\leq &\int_{B_R\setminus
      B_{h_0}}\frac{|g_y-\gd|}{|x|^2}\d\L^2+\frac1{2R}\int_{\partial
      B_R}|g_y-\gd|\d\H^1\\
&+\frac1{2h_0}\int_{\partial
      B_{h_0}}|g_y-\gd|\d\H^1\label{eq:3}
  \end{split}
  \end{equation}
By Cauchy-Schwarz,

\begin{equation}
\begin{split}
  \int_{B_R\setminus B_{h_0}}\frac{|g_y-\gd|}{|x|^2}\d\L^2 \leq
  &\left(\int_{B_R\setminus
      B_{h_0}}|g_y-\gd|^2\d\L^2\right)^{1/2}\left(\int_{B_R\setminus
      B_{h_0}}|x|^{-4}\d\L^2\right)^{1/2}\\
  \leq& \left(\int_{B_R\setminus
      B_{h_0}}|g_y-\gd|^2\d\L^2\right)^{1/2}\sqrt{2\pi}h_0^{-1}\\
  \int_{\partial B_R}|g_y-\gd|\d\H^1\leq &C\sqrt{R}\left(\int_{\partial
      B_R}|g_y-\gd|^2\d\H^1\right)^{1/2}\\
  \int_{\partial B_{h_0}}|g_y-\gd|\d\H^1\leq &C\sqrt{h_0}\left(\int_{\partial
      B_{h_0}}|g_y-\gd|^2\d\H^1\right)^{1/2}\,.
\end{split}\label{eq:33}
\end{equation}
Now choose $R_0\in[R-h,R]$ such that
\[
\int_{\partial B_{R_0}}|g_y-\gd|^2\d\H^1\leq Ch^{-1}\int_{B_1}|g_y-\gd|^2\d\L^2\,.
\]
Together with \eqref{eq:5} and \eqref{eq:33}, \ref{eq:3} becomes

\[
\begin{split}
  |A(R_0)|\leq & C\frac{\Em(y)^{1/2}}{h_0}\,,
  \end{split}
\]
where $\Em(y)$ is the membrane energy,
\[
\Em(y):=\int_{B_1}|g_y-\gd|^2\d\L^2\,.
\]
The lower bound for the bending energy \eqref{eq:18} becomes

\begin{equation}
\int_{B_{R_0}\setminus B_{h_0}}|D^2y|^2\d\L^2\geq 2\pi\Delta^2\log\frac{R_0}{h_0}-
C \frac{\Em(y)^{1/2}}{h_0}\,.\label{eq:7}
\end{equation}

We use \eqref{eq:7} with $R\uparrow 1$ to estimate the membrane
energy  by

\begin{equation}
\begin{split}
 \Em(y)\leq& 2\pi\Delta^2h^2\left(\log
    \frac{1}{h}+C_2\right)-2\pi\Delta^2h^2\log\frac{1}{h_0}+Ch^2\frac{\Em(y)^{1/2}}{h_0}\\
  \leq&  C\left(h^2
  +h\Em(y)^{1/2}\right)\,.
\end{split}\label{eq:34}
\end{equation}
Using Young's inequality $ab\leq \frac12\left((\e a)^2+(b/\e)^2\right)$, with
$\e=C^{-1}$, we have 
\[
\begin{split}
  Ch\Em(y)^{1/2}\leq &\frac12
  \Em(y)+Ch^2\,,
\end{split}
\]
and inserting this in \eqref{eq:34}, we get
\[
\begin{split}
  \Em(y)\leq &C h^2\,,
\end{split}
\]
which proves \eqref{eq:21}. Furthermore, inserting this in \eqref{eq:7}, we have
\[
\int_{B_{R_0}\setminus B_{h_0}}|D^2y|^2\d\L^2\geq 2\pi\Delta^2\log\frac{R_0}{h}-
C\,.
\]
Sending $R\to 1$, this
proves the lower bound in \eqref{eq:22}.
Furthermore,  
\[
\begin{split}
  \int_{B_1\setminus B_R}|D^2y|^2\d \L^2 &\leq
  h^{-2}(\Ihd(y)-\Em(y))-\int_{h_0}^{R_0}|D^2y|^2\d\L^2\\
  &\leq 2\pi\Delta^2(\log
  1/h+C_2)-2\pi\Delta^2\log\frac{R_0}{h}\\
  &\leq 2\pi\Delta^2\log\frac{1}{R}+C\,,
\end{split}
\]
which proves \eqref{eq:19}.
This completes the proof of the theorem.
\end{proof}
\section{Proof of Theorem \ref{thm:coneconv}}
\subsection{Isometric immersions of a singular cone}
\label{sec:isom-immers-sing}
The plan of the proof is as follows: The crucial inequality \eqref{eq:19} shows that on a fixed annulus $B_1\setminus
B_R$, the $W^{2,2}$ norm of a sequence of deformations $y_h$ satisfying
$\Ihd(y_h)\leq 2\pi\Delta^2h^2(\log 1/h+C)$ is bounded as $h\to 0$. One gets
weak convergence in $W^{2,2}$ to a limit deformation that is an isometric
immersion with respect to $\gd$ (since  the membrane energy of the limit
function vanishes by $\Em(y_h)\leq C h^2\to
0$). We may apply the results on $W^{2,2}$ isometric immersions from
\cite{hornung2011approximation,MR2128713} to the limit, which means that the
limit deformation is developable. Using our energy estimates, we can show that
in fact, it must be identical to the singular cone $\ydel$ up to a Euclidean
motion.\\[.2cm] 
The fact  that flat surfaces are locally developable is a classical result from Differential Geometry of surfaces. For functions in $W^{2,2}$, this statement has been proved in \cite{MR2128713,hornung2011approximation,MR2771671}:
\newcommand{\Lb}{\mathbf{L}}
\begin{theorem}[Theorem 2 in \cite{hornung2011approximation}]
\label{thm:dev}
  Let $\Omega\subset\R^2$ with Lipschitz boundary. Let $y\in W^{2,2}(\Omega;\R^3)$ with $Dy^TDy=\id_{2\times 2}$. Then $y\in C^1(\Omega)$ and there exists a set $\Lb_y$ of mutually disjoint closed line segments in $\bar \Omega$ with endpoints on $\partial \Omega$ with the following property:
For every $x\in\Omega$, either $D^2y=0$ in a neighborhood of $x$, or there exists $L\in\Lb_y$ with $x\in L$ and $Dy$ is constant on $L$.
\end{theorem}
We will need a variant of this theorem for functions whose  domain is a singular
cone.

To be able to use Theorem \ref{thm:dev}, we are going to consider the cone in
a flat reference configuration.
Let $\arccos:[-1,1]\to [0,\pi]$ denote the inverse of $\cos:[0,\pi]\to[-1,1]$.
Define 
\[
B_{1,\Delta}:=\left\{x=(x_1,x_2)\in B_1\setminus\{0\}:0\leq \arccos \frac{ x_1}{
|x|}< \sqrt{1-\Delta^2}\pi\right\}\,.
\] 
Let $\R_-:=\{(x_1,0):x_1\leq 0\}$, and let $\varphi:\R^2\setminus\R_-\to \R$
be the angular coordinate satisfying $x=|x|(\cos\varphi(x),\sin\varphi(x))$ with
values in $(-\pi,\pi)$.
We define the map $\iota\equiv\iota_\Delta:\R^2\setminus\R_-\to B_1$  by
\[
\iota(x)=\left(|x|\cos\frac{\varphi(x)}{\sqrt{1-\Delta^2}},|x|\sin\frac{\varphi(x)}{\sqrt{1-\Delta^2}}\right)\,.
\]
For a sketch of $B_{1,\Delta}$ and $\iota_\Delta$, see Figure \ref{fig:iota}.

\begin{figure}[h]
\includegraphics{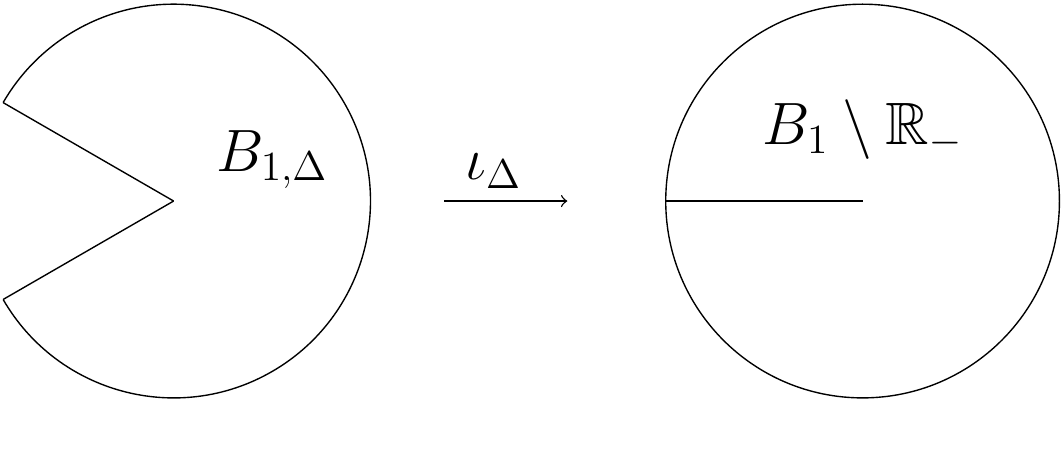}
\caption{The domain $B_{1,\Delta}$ and the map $\iota_\Delta:B_{1,\Delta}\to
  B_1\setminus \R_-$.\label{fig:iota}}
\end{figure}

On $\iota(B_{1,\Delta})=B_1\setminus \R_-$, $\iota$ has a well defined inverse, that we denote by
\[
j: B_1\setminus \R_-\to B_{1,\Delta}\,.
\]
Furthermore, let 
$\phi_\Delta:=(1-\sqrt{1-\Delta^2})2\pi$ and let the rotation $S_\Delta\in
SO(2)$ be defined by
\[
S_{\Delta}=\left(\begin{array}{cc}\cos\phi_\Delta & -\sin\phi_\Delta\\
\sin\phi_\Delta&\cos\phi_\Delta \end{array}\right)\,.
\]
Finally, let 
\[
\begin{split}
  \bdD:=& \partial B_{1,\Delta}\setminus (\partial B_1\cup
  \{0\})
\end{split}
\]
Note that $\bdD$ has two connected components, one contained in the upper half
plane and one in the lower half plane. We will denote them by $\bdD^+$ and
$\bdD^-$ respectively, see Figure \ref{fig:Lad}. The rotation matrix $S_\Delta$
has been chosen such that $S_\Delta\partial_\Delta^+=\partial_\Delta^-$.

\medskip

\newcommand{\Wiso}{W^{2,2}_{\mathrm{iso}}(B_{1,\Delta})}
We define 
\begin{equation}
  \label{eq:32}
  \begin{split}
    \Wiso:=&\Bigg\{Y\in W^{2,2}_{\mathrm{loc}}(B_{1,\Delta};\R^3): g_Y=\id_{2\times 2}\,, \\
    &Y(S_\Delta x)= Y(x)\, \text{ and }\,
    DY(S_\Delta x)=DY( x)S_\Delta \, \text{ for every }\,x\in \bdD^+\,.\Bigg\}
\end{split}
\end{equation}
This definition is chosen such that if
 $y\in W^{2,2}_{\mathrm{loc}}(B_1\setminus\{0\};\R^3)$ with $Dy^TDy=\gd$, then
 $y\circ \iota\in \Wiso$.

\medskip

To $Y\in \Wiso$, 
 we may apply Theorem \ref{thm:dev} (with $\Omega=B_{1,\Delta}\setminus B_\rho$)
 to obtain a set $\Lb_Y$ of line segments with the properties stated there. In
 fact, by sending $\rho\to 0$, we get a set of (relatively) closed mutually
 disjoint line segments in $\overline{B_{1,\Delta}}\setminus\{0\}$. If a line
 segment has only one endpoint in $\overline{B_{1,\Delta}}\setminus\{0\}$, then
 we say by slight abuse of terminology that one of its endpoints is the origin.

\medskip

\newcommand{\Lad}{L^{\mathrm{ad}}}

Next, we are going to define an ``adjoint'' line segment $\Lad$ to any
  $L\in\Lb_Y$ with an endpoint $x\in \bdD$. Note that for such  $L$, there
  exists $v\in \partial B_1$ and $q>0$ such that
\[
L=\{x+tv:t\in [0,q]\}\,.
\]
First let us assume $x\in \partial_\Delta^+$. 
By the definition of $\Wiso$ in \eqref{eq:32},
we have that $x':=S_\Delta x\in \partial_\Delta^-$, and $DY(x')=DY(x)S_\Delta
$. Moreover, there has to 
 exist $\Lad\in \Lb_y$ with $x'\in \Lad$, and $\Lad=\{x'+t S_\Delta
 v:t\in\R\}\cap \overline{B_{1,\Delta}}$. This defines $\Lad$ for
 $x\in\partial_\Delta^+$; for $x\in \partial_\Delta^-$, we define it
 analogously, replacing  $S_\Delta$ by $S_\Delta^{-1}$. For a sketch of the
 construction, see Figure \ref{fig:Lad}.

\begin{figure}[h]
\includegraphics{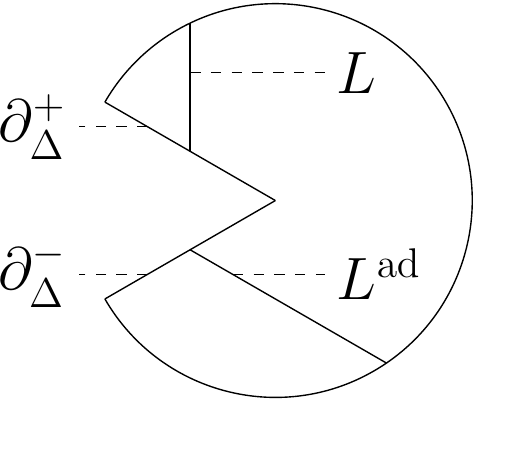}
\caption{The subsets $\partial_\Delta^+,\partial_\Delta^-$ of the boundary and
  adjoint line segments $L,L^{\mathrm{ad}}$. \label{fig:Lad}}
\end{figure}


From now on, the line segments in $\Lb_Y$ for which one of the endpoints is $0$
will be called ``good'', and line segments in the complement of the set of good
line segments will be called ``bad''. The sets of good and bad line segments will be
denoted by $\Lb^{(g)}_Y, \Lb^{(b)}_Y$ respectively. For any bad line segment, we
can lower the elastic energy by ``flattening'' the deformation $Y$ on one side
of the line segment. This is the idea behind the following lemma. For a sketch of
this operation, see Figure \ref{fig:FL}.

\begin{figure}[h]
\includegraphics{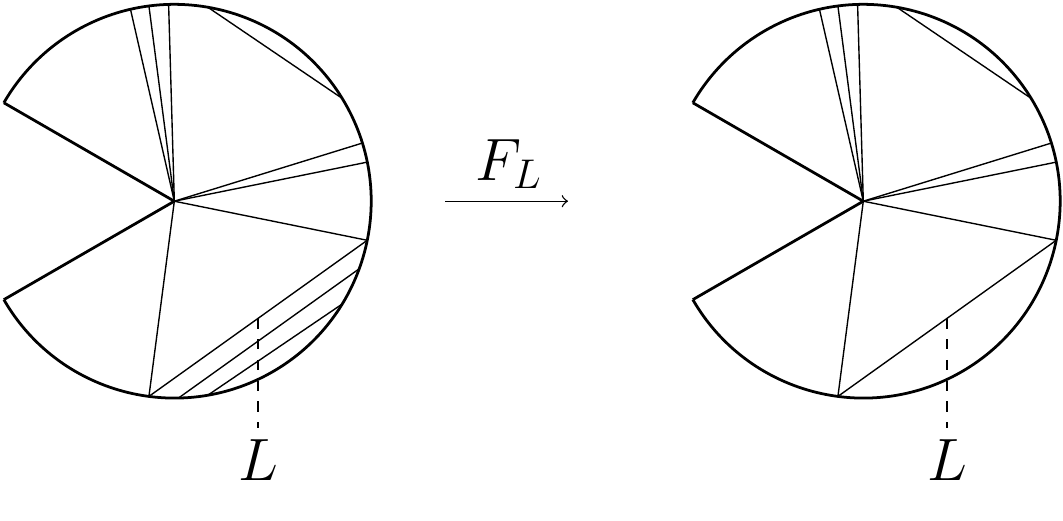}
\caption{
In the left panel, we have the segments that belong to 
 $\Lb_Y$, and $L\in \Lb_Y$  is a bad line segment. We can flatten the
  deformation $Y$ on the side of $L$ whose closure does not contain the origin, and obtain a deformation $F_L(Y)$, such
  that $\Lb_{F_L(Y)}$ consists of those line segments in
  $\Lb_Y$ that are on the same side of $L$ as the origin, see the right panel. \label{fig:FL}}
\end{figure}

\begin{lemma}
\label{lem:flattening}
  For every $Y\in\Wiso$, there exists $Y_\infty\in \Wiso$ with the following
  properties:
  \begin{itemize}
  \item[(i)] $\Lb^{(b)}_{Y_\infty}=\emptyset$ and
    $\Lb^{(g)}_{Y_\infty}=\Lb^{(g)}_{Y}$
\item[(ii)] For $0<\rho<1$, we have 
  \begin{equation}
    \begin{split}
      \int_{B_{1,\Delta}\setminus B_\rho}|D^2Y_\infty:&( (D\iota)^{-1}\otimes
      (D\iota)^{-1})|^2\d\L^2\\
&\leq \int_{B_{1,\Delta}\setminus B_\rho}|D^2Y :(
      (D\iota)^{-1}\otimes (D\iota)^{-1})|^2\d\L^2\,,
    \end{split}
\label{eq:28}
      \end{equation}
    with equality for all $0<\rho<1$ if and only if $Y=Y_\infty$.
  \end{itemize}
\end{lemma}
\begin{proof}
  For any $L\in \Lb^{(b)}_Y$, we may define a modified map $F_L(Y)\in
\Wiso$ as follows.  On $L$, we
  have $Y=A_Lx+b_L$ for some $A_L\in \R^{3\times 2}$ and $b_L\in
  \R^3$.  We note that
  $B_{1,\Delta}\setminus L$ has exactly two connected components. Let
  $E_L$ denote the connected component whose closure does not contain
  the origin. First let us assume that none of the endpoints of $L$ is in
  $\bdD$. Then we define $F_{L}(Y)\in \Wiso$ by

\begin{equation}
  F_{L}(Y)(x)=\begin{cases}A_Lx+b_L &\text{ if }x\in E_L\\Y(x)&\text{ else.}\end{cases}\label{eq:31}
\end{equation}
If one of the endpoints of $L$ is in $\bdD$, then  we set
\begin{equation*}
  F_{L}(Y)(x)=\begin{cases}A_Lx+b_L &\text{ if }x\in E_L\\A_{\Lad}x+b_{\Lad} &\text{ if }x\in E_{\Lad}\\Y(x)&\text{ else.}\end{cases}
\end{equation*}
Note that this definition indeed satisfies $F_{L}(Y)\in \Wiso$.
Obviously, we have $D^2(F_{L}(Y))=0$  on $E_L$ (and on $E_{\Lad}$) and hence, for
all $0<\rho<1$, we have
\begin{equation}
    \int_{B_{1,\Delta}\setminus
      B_\rho}|D^2F_L(Y)(D\iota)^{-1}|^2\d\L^2\leq \int_{B_{1,\Delta}\setminus B_\rho}|D^2Y(D\iota)^{-1}|^2\d\L^2\,\label{eq:36}.
      \end{equation}
We must distinguish two cases in \eqref{eq:36}: If $\Lb_{F_L(Y)}\subsetneq
\Lb_Y$,then  $F_L(Y)\neq Y$ and we must have $|D^2Y|>0$ on a subset of
positive measure of $E_L$. Hence,
 inequality must hold in \eqref{eq:36} for some $\rho$, since we
have
\begin{equation}
\sqrt{1-\Delta^2}\id_{2\times 2}\leq (D\iota)^{-1}\leq \id_{2\times 2}\label{eq:16}
\end{equation}
in the sense of
positive definite matrices. Equality in \eqref{eq:36} only holds in the case $F_L(Y)= Y$.

On $\Lb^{(b)}_Y$, we may
define an order relation by $L< L'$ if $E_L\subsetneq E_{L'}$. Since
bad line segments are mutually disjoint, we have that either $L<L'$,
$L>L'$ or $E_L\cap E_{L'}=\emptyset$. Hence, there exists an at most 
countable sequence $L_1,L_2,\dots$ of maximal bad line segments. If for two maximal line segments $L,L'$
we have $L'=\Lad$  then we exclude exactly one of them from that
sequence.
Now we define a sequence $Y_k\in\Wiso$ by

\begin{equation}
Y_k=F_{L_k}\circ\dots\circ F_{L_1}(Y).\label{eq:25}
\end{equation}

By \eqref{eq:36} and \eqref{eq:16}, $D^2Y_k$
is bounded in $L^2$.  Thus the sequence converges weakly in 
$W^{2,2}(B_{1,\Delta}\setminus B_{\rho};\R^3)$ for every $0<\rho<1$ to a limit
$Y_\infty\in \Wiso$ such that $\Lb_{Y_\infty}$ does not contain any
bad line segments, and $\Lb^{(g)}_{Y_\infty}=\Lb^{(g)}_{Y}$.
The claim \eqref{eq:28} follows from \eqref{eq:36} and the comment  after that equation. This proves the lemma.
\end{proof}

\begin{remark}
\label{rem:jrem}
  Letting $Y,Y_\infty$ as in Lemma \ref{lem:flattening}, we  have
\[
\begin{split}
  Y_{\infty}\circ j \in& W^{2,2}_{\mathrm{loc}}(B_1\setminus \{0\};\R^3)\\
  g_{Y_{\infty}\circ j}=&\gd\,.
\end{split}
\]
Furthermore,
\[
\int_{B_1\setminus B_\rho}|D^2(Y_\infty\circ j)|^2\d\L^2\leq
\int_{B_1\setminus B_\rho}|D^2(Y\circ j)|^2\d\L^2\quad \text{ for every }0<\rho<1\,.
\]

\end{remark}
\begin{proof}
 The first two statements follows immediately from $Y_\infty\in \Wiso$, and  only the inequality requires a
 proof. Let $\nu=Y_{,1}\wedge Y_{,2}/|Y_{,1}\wedge Y_{,2}|$ be the unit
 normal. By $DY^TDY=\id_{2\times 2}$, we have
 $D^2Y\bot DY$. Hence

\begin{equation}
  \begin{split}
    \left|D^2(Y\circ j)\right|^2=&\left|D^2Y:(Dj\otimes Dj)+DY D^2 j\right|^2\\
      =& \left|D^2Y:(Dj\otimes Dj)\right|^2+\left|DY D^2 j\right|^2\\
      =& \left|D^2Y:(Dj\otimes Dj)\right|^2+\left|D^2 j\right|^2\,,\label{eq:27}
    \end{split}
\end{equation}
where we used $DY\in O(2,3)$ in the last equality.
Now the claim of the remark follows from \eqref{eq:28} and a change of variables
in the integrals. 
\end{proof}

\subsection{Proof of Theorem \ref{thm:coneconv}}

\begin{proof}[Proof of Theorem \ref{thm:coneconv}]
To prove the convergence \eqref{eq:35}, it is enough to
prove that for any subsequence of $y_h$, there exists a further
subsequence such that the convergence \eqref{eq:35} holds. Hence, we assume that
we start with an arbitrary subsequence, and may take further subsequences at will.\\
\\
Given $0<R<1$, we may assume that $h\ll R$. Choose $R_0(h)\in [R-h,R]$ such that
\[
\int_{\partial B_{R_0(h)}}|g_{y_h}-\gd|^2\d\H^1\leq h^{-1}\int_{B_1}|g_{y_h}-\gd|^2\d\L^2\,.
\]
By Theorem \ref{thm:main1}, we have 

\begin{equation}
\begin{split}
  \int_{B_1\setminus B_R}|D^2y_h|\d\L^2&\leq
  \int_{B_1\setminus B_{R_0}}|D^2y_h|\d\L^2\\
  &\leq 2\pi\Delta^2\log \frac1{R}+C\,,
\end{split}\label{eq:8}
\end{equation}
where $C$ does neither depend on $h$ nor on $R$. This proves the boundedness of $y_h$ in
$W^{2,2}(B_1\setminus B_R;\R^3)$ and implies that there exists $\hat y_R\in
W^{2,2}(B_1\setminus B_R;\R^3)$ such that (for a subsequence)
\[
y_h\wto \hat y_R \quad \text{ in }W^{2,2}(B_1\setminus B_R;\R^3)\,.
\]
After taking a suitable diagonal sequence for $R=\frac1j$, $j=2,3,\dots$, we may
assume that $\hat y_R\in W^{2,2}_{\mathrm{loc.}}(B_1\setminus\{0\};\R^3)$ is
independent of $R$. We denote this function by $y^*$. By Theorem
\ref{thm:main1}, we have
\[
\int_{B_1}|g_{y^*}-\gd|\d\L^2=0\,.
\]
I.e., $ y^*$ is an isometry with respect to $\gd$. \\
By \eqref{eq:8},  we have
\begin{equation}
\int_{B_1\setminus B_R}|D^2 y^*|^2\d\L^2\leq 2\pi\Delta^2h^2\log \frac1R+C\,.\label{eq:9}
\end{equation}
Let $Y:B_{1,\Delta}\to\R^3$ be defined by
\[
Y:=y^*\circ \iota\,.
\]
Recalling the definitions from Section \ref{sec:isom-immers-sing}, we have
 $Y\in \Wiso$. By an application of Lemma \ref{lem:flattening}, we obtain $Y_\infty\in \Wiso$  such
 that $DY_\infty$ is constant on every line segment $(0,x)$ with $x\in \partial
 B_{1,\Delta}\cap \partial B_1$.
Now we set $y^\infty:=Y_\infty\circ j$, and obtain that $Dy^\infty$ is constant on every
line segment $(0,x]$ with $x\in \partial B_1$. Hence there exists a curve
$\gamma:\partial B_1\to S^2$ satisfying
$|\gamma'|=\sqrt{1-\Delta^2}$ such that
\begin{equation}
 y^\infty(x)=x \gamma(x/|x|)\,.\label{eq:15}
\end{equation}
Using this expression, explicit computation yields

\begin{equation}
\int_{\partial B_\rho}|D^2y^\infty|^2\d\H^1=\frac{1}{\rho}\int_{\partial
  B_1}|D^2y^\infty|^2\d\H^1\,.\label{eq:29}
\end{equation}
By Remark \ref{rem:jrem} and \eqref{eq:9},  we have that for every $0<\rho<1$,
  \begin{equation}
\int_{B_1\setminus B_{\rho}}|D^2 y^\infty|^2\d\L^2\leq 2\pi\Delta^2\log\frac{1}{\rho}+C\,. \label{eq:10}
\end{equation}

Combining \eqref{eq:29} and \eqref{eq:10}, we see that for every $0<\rho<1$, we have
\[
\int_{\partial B_\rho}|D^2 y^\infty|^2\d\H^1\leq \frac{2\pi\Delta^2}{\rho}\,,
\]
and the constant $C$ in \eqref{eq:10} is in fact 0.\\
\\
By $g_{y^\infty}=\gd$, we have
\[
\sum_{i=1}^3\det D^2 y^\infty_i=\pi\Delta^2\delta_0
\]
distributionally. We may now estimate using Lemma \ref{lem:isoper}, for any
$0<\rho<1$, 

\begin{equation}
\begin{split}
  \pi\Delta^2=&\int_{B_\rho}\sum_{i=1}^3\det D^2 y^\infty_i\d \L^2\\
  \leq &\sum_i\left|\int_{B_\rho}\det D^2 y^\infty_i\d \L^2\right|\\
\leq &\frac{1}{4\pi}\sum_i\left(\int_{\partial B_\rho} |D^2y^\infty_i(x)\cdot \hat x^\bot|\d\H^1(x)\right)^2\\
  \leq &\frac{1}{4\pi}\sum_i2\pi \rho\left(\int_{\partial B_\rho}|D^2y^\infty_i(x)\cdot \hat
    x^\bot|^2\d\H^1(x)\right)\\
\leq & \frac{\rho}{2}\int_{\partial B_\rho}|D^2 y^\infty(x)\cdot \hat
    x^\bot|^2\d\H^1(x)\\
\leq&\pi\Delta^2\,.
\end{split}\label{eq:11}
\end{equation}
Here, to obtain the fourth from the third line, we have used Jensen's inequality.
By this chain of estimates, all the inequalities must have been equalities, and we  have
\[
\left(\int_{\partial B_\rho}\sum_i |D^2y^\infty_i(x)\cdot \hat
  x^\bot|\d\H^1(x)\right)^2=2\pi \rho \left(\int_{\partial B_\rho}|D^2y^\infty_i(x)\cdot \hat
    x^\bot|^2\d\H^1(x)\right)
\]
and thus
\begin{equation}
|D^2y^\infty_i(x)\cdot \hat x^\bot|^2=\mathrm{constant}\quad \text{ for
}x\in\partial B_\rho,\,i\in\{1,2,3\}\,.\label{eq:13}
\end{equation}
Additionally, \eqref{eq:11} implies that
\begin{equation}
  |D^2y^\infty(x)\cdot \hat x^\bot|^2=\frac{\Delta^2}{\rho^2}\quad \text{ for }x\in\partial B_\rho\,.\label{eq:14}
\end{equation}
By \eqref{eq:15}, we have $D^2y^\infty(x)=|x|^{-1}(\gamma+\gamma'')\otimes\hat
x^\bot\otimes\hat x^\bot$. Combining this with \eqref{eq:13}, we get
\[
(\gamma+\gamma'')\cdot e_i=\mathrm{constant} \text{ on }\partial B_1
\]
for $i=1,2,3$. We write $c_i= (\gamma+\gamma'')\cdot e_i$, and have
 $D^2y^\infty_i(x)=\frac{c_i}{|x|}\hat x^\bot\otimes\hat x^\bot$, which
implies
\[
y^\infty_i(x)=c_i|x|+a_i\cdot x+b_i \quad \text{ for }i=1,2,3\,,
\]
for some  $a_i\in \R^2,\,b_i\in\R$. By \eqref{eq:13} we obtain
\[
\left|D^2 y^\infty(x)\right|^2=\frac{\sum_i
  c_i^2}{|x|^2}=\frac{\Delta^2}{|x|^2}\,,
\]
and thus
$\sum_{i}c_i^2=\Delta^2$. By $g_{y^\infty}=g_{\Delta}$, we have
\[
\begin{split}
  \id_{2\times 2}-\Delta^2\hat x^\bot\otimes\hat x^\bot=&(c\otimes \hat
  x+a)^T(c\otimes x+a)\\
  =& |c|^2\hat x\otimes \hat x+(c\cdot a)\otimes\hat x+\hat x\otimes(c\cdot
  a)+a^Ta\,.
\end{split}
\]
This yields 
\[
(1-\Delta^2)\id_{2\times 2}= (c\cdot a)\otimes\hat x+\hat x\otimes(c\cdot
  a)+a^Ta\,,
\]
which can only hold true for all $\hat x\in \partial B_1$ if  $c\cdot a=0$ and
$a^Ta=(1-\Delta^2)\id_{2\times 2}$. This implies that
\[
R:=\left(\frac{a}{\sqrt{1-\Delta^2}},\frac{c}{\Delta}\right)\in O(3)
\]
is an orthogonal matrix, and we have
\[
y^\infty(x)=R\left(\sqrt{1-\Delta^2}x+\Delta e_3 |x|\right)+b\,.
\]
It remains to show that $y^\infty=y^*$.
To see this, note that $y^\infty\circ\iota=Y_\infty$ satisfies
\[
\left\{(0,x]:x\in \partial B_{1,\Delta}\cap \partial B_1\right\} =\Lb_{Y_\infty}^{(g)}=\Lb_{Y}^{(g)}\,,
\]
where the second equality holds by Lemma \ref{lem:flattening}.
This implies that for every $x\in
B_{1,\Delta}$ there exists an $L\in \Lb^{(g)}_Y$ with  $x\in L$. This in turn
implies that $\Lb_Y^{(b)}=\emptyset$ (since the line segments in $\Lb_Y$ are
pairwise disjoint). By Lemma \ref{lem:flattening}, the latter
yields $Y=Y_\infty$. Composing with $j$ on both sides of this last equation, we
obtain $y^*=y^\infty$. This completes the proof of the theorem.
\end{proof}


\bibliographystyle{plain}
\bibliography{regular}

\end{document}